\documentclass[10pt]{article}
\usepackage{amsmath}
\usepackage{amssymb}
\usepackage{amsthm}
\usepackage{amscd}
\usepackage{xypic}
\usepackage{delarray}
\usepackage{mathrsfs}

\headsep=1cm \numberwithin{equation}{section}
\def\C{{\mathbb C}}

\def\F{{\mathbb F}}

\def\Q{{\mathbb Q}}
\def\R{{\mathbb R}}

\def\Z{{\mathbb Z}}

\newtheorem{theorem}{Theorem}[section]
\newtheorem{lemma}[theorem]{Lemma}
\newtheorem{proposition}[theorem]{Proposition}
\newtheorem{corollary}[theorem]{Corollary}
\newtheorem{definition and lemma}[theorem]{Definition and Lemma}

\theoremstyle{definition}
\newtheorem{definition}[theorem]{Definition}

\theoremstyle{remark}
\newtheorem{remark}[theorem]{Remark}

\numberwithin{equation}{section}

\date{}

\begin{document}
\author{WonTae Hwang}

\title{Automorphism groups of simple polarized abelian varieties of odd prime dimension over finite fields}

\maketitle

\begin{abstract}
We prove that the automorphism groups of simple polarized abelian varieties of odd prime dimension over finite fields are cyclic, and give a complete list of finite groups that can be realized as such automorphism groups.
\end{abstract}

\section{Introduction}
Let $k$ be a finite field, and let $X$ be an abelian variety of dimension $g$ over $k.$ We denote the endomorphism ring of $X$ over $k$ by $\textrm{End}_k(X)$. It is a free $\Z$-module of rank $\leq 4g^2.$ We also let $\textrm{End}^0_k(X)=\textrm{End}_k(X) \otimes_{\Z} \Q.$ This $\Q$-algebra $\textrm{End}_k^0(X)$ is called the endomorphism algebra of $X$ over $k.$ Then $\textrm{End}_k^0(X)$ is a finite dimensional semisimple algebra over $\Q$ with $\textrm{dim}_{\Q} \textrm{End}_k^0(X) \leq 4g^2.$ Moreover, if $X$ is simple, then $\textrm{End}_k^0(X)$ is a division algebra over $\Q$. Now, it is well-known that $\textrm{End}_k (X)$ is a $\Z$-order in $\textrm{End}_k^0(X).$ The group $\textrm{Aut}_k(X)$ of the automorphisms of $X$ over $k$ is not finite, in general. But if we fix a polarization $\mathcal{L}$ on $X$, then the group $\textrm{Aut}_k(X,\mathcal{L})$ of the automorphisms of the polarized abelian variety $(X,\mathcal{L})$ is finite. The goal of this paper is to classify all finite groups that can be realized as the automorphism group $\textrm{Aut}_k(X,\mathcal{L})$ of a simple polarized abelian variety $(X,\mathcal{L})$ of odd prime dimension over a finite field $k.$ \\

Along this line, the author \cite{6} gave such a classification for arbitrary abelian surfaces over finite fields. In this paper, as next step toward goal of considering higher dimensional cases, we will prove that the automorphism groups of simple polarized abelian varieties of odd prime dimension over finite fields are cyclic, and give an explicit description of such groups.  \\

One of our main results is the following
\begin{theorem}\label{main theorem1}
  The automorphism groups of simple polarized abelian varieties $(X, \mathcal{L})$ of odd prime dimension over finite fields are cyclic.
\end{theorem}
More precisely, we will prove the following result.
\begin{theorem}\label{main theorem2}
  The possibilities for automorphism groups of a simple polarized abelian variety $(X, \mathcal{L})$ of odd prime dimension over a finite field is given by Table 2 below.
\end{theorem}
For more details, see Corollary \ref{cor 19} and Theorem \ref{main thm}. \\

This paper is organized as follows: In Section~\ref{prelim}, we introduce several facts which are related to our desired classification. Explicitly, we will recall some facts about endomorphism algebras of simple abelian varieties ($\S$\ref{end alg av}), the theorem of Tate ($\S$\ref{thm Tate sec}), Honda-Tate theory ($\S$\ref{thm Honda}), and a result of Waterhouse ($\S$\ref{thm waterhouse}). In Section~\ref{findiv}, we find all the finite groups that can be embedded in certain division algebras using a paper of Amitsur \cite{1}. In Section~\ref{main}, we finally obtain the desired classification using the facts that were introduced in the previous sections, and give another interesting result (see Theorem \ref{application} below). \\

In the sequel, let $g \geq 3$ be a prime number, and let $q=p^a$ for some prime number $p$ and an integer $a \geq 1,$ unless otherwise stated. Also, let $\overline{k}$ denote an algebraic closure of $k.$

\section{Preliminaries}\label{prelim}
In this section, we recall some of the facts in the general theory of abelian varieties over a field. Our main references are \cite{2} and \cite{8}.

\subsection{Endomorphism algebras of simple abelian varieties of odd prime dimension over finite fields}\label{end alg av}
In this section, we classify all the possible endomorphism algebras of simple abelian varieties of odd prime dimension over finite fields. \\

Let $X$ be a simple abelian variety of dimension $g$ over a finite field $k.$ Then it is well-known that $\textrm{End}_k^0(X)$ is a division algebra over $\Q$ with $2g \leq \textrm{dim}_{\Q} \textrm{End}_k^0(X) < (2g)^2.$ Before giving our first result, we also recall Albert's classification. We choose a polarization $\lambda : X \rightarrow \widehat{X}$ where $\widehat{X}$ denotes the dual abelian variety of $X$. Using the polarization $\lambda,$ we can define an involution, called the \emph{Rosati involution}, $^{\vee}$ on $\textrm{End}_k^0(X).$ (For a more detailed discussion about the Rosati involution, see \cite[\S20]{8}.) In this way, to the pair $(X,\lambda)$ we associate the pair $(D, ^{\vee})$ with $D=\textrm{End}_k^0(X)$ and $^{\vee}$, the Rosati involution on $D$. We know that $D$ is a simple division algebra over $\Q$ of finite dimension and that $^{\vee}$ is a positive involution. Let $K$ be the center of $D$ so that $D$ is a central simple $K$-algebra, and let $K_0 = \{ x \in K~|~x^{\vee} = x \}$ be the subfield of symmetric elements in $K.$ By a theorem of Albert, the pair $(D,^{\vee})$ is of one of the following four types: \\

(i) Type I: $K_0 = K=D$ is a totally real field and $^{\vee} = \textrm{id}_D.$ \\

(ii) Type II: $K_0 = K$ is a totally real field, and $D$ is a quaternion algebra over $K$ with $D\otimes_{K,\sigma} \R \cong M_2(\R)$ for every embedding $\sigma : K \hookrightarrow \R.$ Now, let $a \mapsto \textrm{Trd}_{D/K}(a)-a$ be the canonical involution on $D.$ Then there is an element $b \in D$ such that $b^2 \in K$ is totally negative, and such that $a^{\vee}=b (\textrm{Trd}_{D/K}(a)-a) b^{-1}$ for all $a \in D.$ We have an isomorphism $D \otimes_{\Q}\R \cong \prod_{\sigma : K \hookrightarrow \R} M_2(\R)$ such that the involution $^{\vee}$ on $D \otimes_{\Q} \R$ corresponds to the involution $(A_1, \cdots, A_e) \mapsto (A_1^t , \cdots, A_e^t)$ where $e=[K :\Q].$ \\

(iii) Type III: $K_0 = K$ is a totally real field, and $D$ is a quaternion algebra over $K$ with $D\otimes_{K,\sigma} \R \cong \mathbb{H}$ for every embedding $\sigma : K \hookrightarrow \R$ (where $\mathbb{H}$ is the Hamiltonian quaternion algebra over $\R$). Now, let $a \mapsto \textrm{Trd}_{D/K}(a)-a$ be the canonical involution on $D.$ Then $^{\vee}$ is equal to the canonical involution on $D.$ We have an isomorphism $D \otimes_{\Q}\R \cong \prod_{\sigma : K \hookrightarrow \R} \mathbb{H}$ such that the involution $^{\vee}$ on $D \otimes_{\Q} \R$ corresponds to the involution $(\alpha_1, \cdots, \alpha_e) \mapsto (\overline{\alpha_1} , \cdots, \overline{\alpha_e})$ where $e=[K :\Q].$ \\

(iv) Type IV: $K_0 $ is a totally real field, $K$ is a totally imaginary quadratic field extension of $K_0$. Write $a \mapsto \overline{a}$ for the unique non-trivial automorphism of $K$ over $K_0$; this automorphism is usually referred to as complex conjugation. If $\nu$ is a finite place of $K$, write $\overline{\nu}$ for its complex conjugate. The algebra $D$ is a central simple algebra over $K$ such that: (a) If $\nu$ is a finite place of $K$ with $\nu = \overline{\nu},$ then $\textrm{inv}_{\nu}(D)=0$; (b) For any place $\nu$ of $K$, we have $\textrm{inv}_{\nu}(D)+\textrm{inv}_{\overline{\nu}}(D)=0$ in $\Q/\Z.$ If $d$ is the degree of $D$ as a central simple $K$-algebra, then we have an isomorphism $D \otimes_{\Q}\R \cong \prod_{\sigma : K_0 \hookrightarrow \R} M_d(\C)$ such that the involution $^{\vee}$ on $D \otimes_{\Q} \R$ corresponds to the involution $(A_1, \cdots, A_{e_0}) \mapsto (\overline{A_1}^t , \cdots, \overline{A_{e_0}}^t)$ where $e_0=[K_0 :\Q].$ \\

Keeping the notations as above, we let
\begin{equation*}
  e_0= [K_0 :\Q],~~~e=[K:\Q],~~~\textrm{and}~~~d=[D:K]^{\frac{1}{2}}.
\end{equation*}

As our last preliminary fact of this section, we impose some numerical restrictions on those values $e_0, e,$ and $d$ in the next table, following \cite[\S21]{8}.
\begin{center}
  \begin{tabular}{|c|c|c|}
\hline
$$ & $\textrm{char}(k)=0$ & $\textrm{char}(k)=p>0$ \\
\hline
$\textrm{Type I}$ & $e|g$ & $e|g$  \\
\hline
$\textrm{Type II}$ & $2e|g$  & $2e|g$ \\
\hline
$\textrm{Type III}$ & $2e|g$ & $e|g$ \\
\hline
$\textrm{Type IV}$ & $e_0 d^2 |g$ &$ e_0 d |g $ \\
\hline
\end{tabular}
\vskip 4pt
\textnormal{Table 1}
\end{center}

Now, we are ready to introduce the following
\begin{lemma}\label{poss end alg}
Let $X$ be a simple abelian variety of dimension $g$ over a finite field $k=\F_q$, and let $\lambda : X \rightarrow \widehat{X}$ be a polarization. Then $D:=\textrm{End}_k^0(X)$ (together with the Rosati involution $^\vee$ corresponding to $\lambda$) is of one of the following two types: \\
(1) $D$ is a totally imaginary quadratic extension field of a totally real field of degree $g$; \\
(2) $D$ is a central simple division algebra of degree $g$ over a totally imaginary quadratic field.
\end{lemma}
\begin{proof}
First, we recall that $X$ is of CM-type (see Corollary \ref{cor TateEnd0} below), and hence, either $D$ is of Type III or Type IV in Albert's classification. It suffices to show that $D$ is not of Type III. To this aim, suppose that $D$ is of Type III. By Table 1, we know that $e|g.$ If $e=1,$ then $\textrm{dim}_{\Q} D =4,$ which contradicts the fact that $2g \leq \textrm{dim}_{\Q} D$ and $g \geq 3.$ If $e=g,$ then $D$ is a quaternion algebra over a totally real field $K=\Q(\pi_X)$ of degree $g$ where $\pi_X$ denotes the Frobenius endomorphism of $X.$ Since $\pi_X$ is a $q$-Weil number (see Remark \ref{qWeil rem} below), we have $|\pi_X |=\sqrt{q}$ so that $[K : \Q] \leq 2,$ which is also a contradiction. Hence, we can conclude that $D$ cannot be of Type III. \\

This completes the proof.
\end{proof}

\begin{remark}\label{oort rmk}
Let $k$ be an algebraically closed field with $\textrm{char}(k)=p>0.$ Let $X$ be a simple abelian variety of dimension $g$ over $k$, and let $D=\textrm{End}_k^0(X)$. Then $D$ is of one of the following types (see \cite[$\S7$]{10}): \\
(i) $D=\Q$;\\
(ii) $D$ is a totally real field of degree $g$;\\
(iii) $D=D_{p,\infty}$ if $g \geq 5$; \\
(iv) $D$ is a totally imaginary quadratic field;\\
(v) $D$ is a totally imaginary quadratic extension field of a totally real field of degree $g$; \\
(vi) $D$ is a central simple division algebra of degree $g$ over a totally imaginary quadratic field and the $p$-rank of $X$ is $0.$
\end{remark}

\subsection{The theorem of Tate}\label{thm Tate sec}
In this section, we recall an important theorem of Tate, and give some interesting consequences of it. \\

Let $k$ be a field and let $l$ be a prime number with $l \ne \textrm{char}(k)$. If $X$ is an abelian variety of dimension $g$ over $k,$ then we can introduce the Tate $l$-module $T_l X$ and the corresponding $\Q_l$-vector space $V_l X :=T_l X \otimes_{\Z_l} \Q_l.$ It is well-known that $T_l X$ is a free $\Z_l$-module of rank $2g$ and $V_l X$ is a $2g$-dimensional $\Q_l$-vector space. In \cite{11}, Tate showed the following important result:

\begin{theorem}\label{thm Tate}
Let $k$ be a finite field and let $\Gamma = \textrm{Gal}(\overline{k}/k).$ If $l$ is a prime number with $l \ne \textrm{char}(k),$ then we have: \\
(a) For any abelian variety $X$ over $k,$ the representation
\begin{equation*}
 \rho_l =\rho_{l,X} : \Gamma \rightarrow \textrm{GL}(V_l X)
\end{equation*}
is semisimple. \\
(b) For any two abelian varieties $X$ and $Y$ over $k,$ the map
\begin{equation*}
 \Z_l \otimes_{\Z} \textrm{Hom}_k(X,Y) \rightarrow \textrm{Hom}_{\Gamma}(T_l X, T_l Y)
\end{equation*}
is an isomorphism.
\end{theorem}

Now, we recall that an abelian variety $X$ over a (finite) field $k$ is called \emph{elementary} if $X$ is $k$-isogenous to a power of a simple abelian variety over $k.$ Then, as an interesting consequence of Theorem \ref{thm Tate}, we have the following
\begin{corollary}\label{cor TateEnd0}
  Let $X$ be an abelian variety of dimension $g$ over a finite field $k.$ Then we have:\\
  (a) The center $Z$ of $\textrm{End}_k^0(X)$ is the subalgebra $\Q[\pi_X].$ In particular, $X$ is elementary if and only if $\Q[\pi_X]=\Q(\pi_X)$ is a field, and this occurs if and only if $f_X$ is a power of an irreducible polynomial in $\Q[t]$ where $f_X$ denotes the characteristic polynomial of $\pi_X.$ \\
 (b) Suppose that $X$ is elementary. Let $h=f_{\Q}^{\pi_X}$ be the minimal polynomial of $\pi_X$ over $\Q$. Further, let $d=[\textrm{End}_k^0(X):\Q(\pi_X)]^{\frac{1}{2}}$ and $e=[\Q(\pi_X):\Q].$ Then $de =2g$ and $f_X = h^d.$ \\
(c) We have $2g \leq \textrm{dim}_{\Q} \textrm{End}^0_k (X) \leq (2g)^2$ and $X$ is of CM-type. \\
(d) The following conditions are equivalent: \\
  \indent (d-1) $\textrm{dim}_{\Q} \textrm{End}_k^0(X)=2g$; \\
  \indent (d-2) $\textrm{End}_k^0(X)=\Q[\pi_X]$; \\
  \indent (d-3) $\textrm{End}_k^0(X)$ is commutative; \\
  \indent (d-4) $f_X$ has no multiple root. \\
(e) The following conditions are equivalent: \\
  \indent (e-1) $\textrm{dim}_{\Q} \textrm{End}_k^0(X)=(2g)^2$; \\
  \indent (e-2) $\Q[\pi_X]=\Q$; \\
  \indent (e-3) $f_X$ is a power of a linear polynomial; \\
  \indent (e-4) $\textrm{End}^0_k(X) \cong M_g(D_{p,\infty})$ where $D_{p,\infty}$ is the unique quaternion algebra over $\Q$ that is ramified at $p$ and $\infty$, and split at all other primes; \\
  \indent (e-5) $X$ is supersingular with $\textrm{End}_k(X) = \textrm{End}_{\overline{k}}(X_{\overline{k}})$ where $X_{\overline{k}}=X \times_k \overline{k}$; \\
  \indent (e-6) $X$ is isogenous to $E^g$ for a supersingular elliptic curve $E$ over $k$ all of whose endomorphisms are defined over $k.$
\end{corollary}
\begin{proof}
  For a proof, see \cite[Theorem 2]{11}.
\end{proof}

For a precise description of the structure of the endomorphism algebra of an elementary abelian variety $X$, viewed as a simple algebra over its center $\Q[\pi_X]$, we record the following two results:
\begin{proposition}\label{local inv}
  Let $X$ be an elementary abelian variety over a finite field $k=\F_q.$ Let $K=\Q[\pi_X].$ If $\nu$ is a place of $K$, then the local invariant of $\textrm{End}_k^0(X)$ in the Brauer group $\textrm{Br}(K_{\nu})$ is given by
  \begin{equation*}
    \textrm{inv}_{\nu}(\textrm{End}_k^0(X))=\begin{cases} 0 & \mbox{if $\nu$ is a finite place not above $p$}; \\ \frac{\textrm{ord}_{\nu}(\pi_X)}{\textrm{ord}_{\nu}(q)} \cdot [K_{\nu}:\Q_p] & \mbox{if $\nu$ is a place above $p$}; \\ \frac{1}{2} & \mbox{if $\nu$ is a real place of $K$}; \\ 0 & \mbox{if $\nu$ is a complex place of $K$}. \end{cases}
  \end{equation*}
\end{proposition}
\begin{proof}
  For a proof, see \cite[Corollary 16.30]{2}.
\end{proof}

\begin{proposition}\label{index end alg}
  Let $X$ be a simple abelian variety over a finite field $k.$ Let $d$ be the degree of the division algebra $D:=\textrm{End}_k^0(X)$ over its center $\Q(\pi_X)$ (so that $d=[D:\Q(\pi_X)]^{\frac{1}{2}}$ and $f_X = (f_{\Q}^{\pi_X})^d$). Then $d$ is the least common denominator of the local invariants $\textrm{inv}_{\nu}(D).$
\end{proposition}
\begin{proof}
  For a proof, see \cite[Corollary 16.32]{2}.
\end{proof}

\subsection{Abelian varieties up to isogeny and Weil numbers: Honda-Tate theory}\label{thm Honda}
In this section, we recall an important theorem of Honda and Tate. To achieve our goal, we first give the following
\begin{definition}\label{qWeil Def}
  (a) A \emph{$q$-Weil number} is an algebraic integer $\pi$ such that $| \iota(\pi) | = \sqrt{q}$ for all embeddings $\iota : \Q[\pi] \hookrightarrow \C.$ \\
  (b) Two $q$-Weil numbers $\pi$ and $\pi^{\prime}$ are said to be \emph{conjugate} if they have the same minimal polynomial over $\Q,$ or equivalently, there is an isomorphism $\Q[\pi] \rightarrow \Q[\pi^{\prime}]$ sending $\pi$ to $\pi^{\prime}.$
\end{definition}

Regarding $q$-Weil numbers, the following facts are well-known:
\begin{remark}\label{qWeil rem}
Let $X$ and $Y$ be abelian varieties over a finite field $k=\F_q.$ Then we have \\
(1) The Frobenius endomorphism $\pi_X$ is a $q$-Weil number. \\
(2) $X$ and $Y$ are $k$-isogenous if and only if $\pi_X$ and $\pi_Y$ are conjugate.
\end{remark}

Now, we introduce our main result of this section:

\begin{theorem}\label{thm HondaTata}
 For every $q$-Weil number $\pi$, there exists a simple abelian variety $X$ over $\F_q$ such that $\pi_X$ is conjugate to $\pi$. Moreover, we have a bijection between the set of isogeny classes of simple abelian varieties over $\F_q$ and the set of conjugacy classes of $q$-Weil numbers given by $X \mapsto \pi_X$.
\end{theorem}
The inverse of the map $X \mapsto \pi_X$ associates to a $q$-Weil number $\pi$ a simple abelian variety $X$ over $\F_q$ such that $f_X$ is a power of the minimal polynomial $f_{\Q}^{\pi}$ of $\pi$ over $\Q.$
\begin{proof}
For a proof, see \cite[Main Theorem]{4} or \cite[\S16.5]{2}.
\end{proof}

\subsection{Isomorphism classes contained in an isogeny class}\label{thm waterhouse}
In this section, we will give a useful result of Waterhouse~\cite{12}. Throughout this section, let $k=\F_q.$ \\

Let $X$ be an abelian variety over $k.$ Then $\textrm{End}_k(X)$ is a $\Z$-order in $\textrm{End}_k^0(X)$ containing $\pi_X$ and $q/\pi_X.$ If a ring is the endomorphism ring of an abelian variety, then we may consider a left ideal of the ring, and give the following
\begin{definition}\label{Roccur}
  Let $X$ be an abelian variety over $k$ with $R:=\textrm{End}_k(X),$ and let $I$ be a left ideal of $R$. \\
  (a) We define $H(X,I)$ to be the intersection of the kernels of all elements of $I$. This is a finite subgroup scheme of $X.$ \\
  (b) We define $X_I$ to be the quotient of $X$ by $H(X,I)$ i.e.\ $X_I=X/H(X,I).$ This is an abelian variety over $k$ that is $k$-isogenous to $X.$
\end{definition}

Now, to introduce our main result of this section, we need the following
\begin{lemma}\label{endchange}
  Let $X$ be an abelian variety over $k$ with $R:=\textrm{End}_k(X)$, and let $I$ be a left ideal of $R.$ Then $\textrm{End}_k(X_I)$ contains $O_r (I):=\{x \in \textrm{End}_k^0(X)~|~I x \subseteq I\}$, the right order of $I,$ and equals it if $I$ is a kernel ideal.
\end{lemma}
\begin{proof}
For a proof, see \cite[Lemma 16.56]{2} or \cite[Proposition 3.9]{12}.
\end{proof}

Using this lemma, we can prove the following result, which plays an important role:
\begin{proposition}\label{endposs}
  Let $X$ be an abelian variety over $k$. Then every maximal order in $D:=\textrm{End}_k^0(X)$ occurs as the endomorphism ring of an abelian variety in the isogeny class of $X.$
\end{proposition}
\begin{proof}
For a proof, see \cite[Theorem 3.13]{12}.
\end{proof}

\section{Finite subgroups of division algebras}\label{findiv}
Recall that if $X$ is a simple abelian variety of dimension $g$ over a finite field $k$, then $\textrm{End}_k^0(X)$ is either a CM-field of degree $2g$ over $\Q$ or a central simple division algebra of degree $g$ over a totally imaginary quadratic field by Lemma \ref{poss end alg}. Hence, in this section, we give a classification of all possible finite groups that can be embedded in the multiplicative subgroup of a division algebra with certain properties that are related to our situation. Our main reference is a paper of Amitsur~\cite{1}. We start with the following

\begin{definition}\label{def 9}
  Let $m, r$ be two relatively prime positive integers, and we put $s:=\gcd(r-1, m)$ and $t:=\frac{m}{s}.$ Also, let $n$ be the smallest integer such that $r^n \equiv 1~(\textrm{mod}~m).$ We denote by $G_{m,r}$ the group generated by two elements $a,b$ satisfying the relations
  \begin{equation*}
    a^m=1,~b^n=a^t,~ bab^{-1}=a^r.
  \end{equation*}
  This type of groups includes the \emph{dicyclic group} of order $mn,$ in which case, we often write $\textrm{Dic}_{mn}$ for $G_{m,r}$. As a convention, if $r=1,$ then we put $n=s=1,$ and hence, $G_{m,1}$ is a cyclic group of order $m.$
\end{definition}

  Given $m,r,s,t,n,$ as above, we will consider the following two conditions in the sequel: \\
  (C1) $\gcd(n,t)=\gcd(s,t)=1.$ \\
  (C2) $n=2n^{\prime}, m=2^{\alpha} m^{\prime}, s=2 s^{\prime}$ where $\alpha \geq 2,$ and $n^{\prime}, m^{\prime}, s^{\prime}$ are all odd integers. Moreover, $\gcd(n,t)=\gcd(s,t)=2$ and $r \equiv -1~(\textrm{mod}~2^{\alpha}).$ \\

Now, let $p$ be a prime number that divides $m.$ We define: \\
(i) $\alpha_p$ is the largest integer such that $p^{\alpha_p}~|~m.$ \\
(ii) $n_p$ is the smallest integer satisfying $r^{n_p} \equiv 1~(\textrm{mod}~mp^{-\alpha_p}).$ \\
(iii) $\delta_p$ is the smallest integer satisfying $p^{\delta_p} \equiv 1~(\textrm{mod}~mp^{-\alpha_p}).$\\

Then we have the following theorem that provides us with a useful criterion for a group $G_{m,r}$ to be embedded in a division ring:
\begin{theorem}\label{thm 11}
  A group $G_{m,r}$ can be embedded in a division ring if and only if either (C1) or (C2) holds, and one of the following conditions holds: \\
  (1) $n=s=2$ and $r \equiv -1~(\textrm{mod}~m)$. \\
  (2) For every prime $q~|~n,$ there exists a prime $p~|~m$ such that $q \nmid n_p$ and that either \\
  (a) $p \ne 2$, and $\gcd(q,(p^{\delta_p}-1)/s)=1$, or \\
  (b) $p=q=2,$ (C2) holds, and $m/4 \equiv \delta_p \equiv 1~(\textrm{mod}~2).$
\end{theorem}
\begin{proof}
  For a proof, see \cite[Theorem 3, Theorem 4, and Lemma 10]{1}.
\end{proof}

Now, let $G$ be a finite group. One of our main tools in this section is the following result:
\begin{theorem}\label{thm 12}
  $G$ can be embedded in a division ring if and only if $G$ is one of the following types: \\
  (1) Cyclic groups. \\
  (2) $G_{m,r}$ where the integers $m,r,$ etc, satisfy Theorem \ref{thm 11} (which is not cyclic). \\
  (3) $\mathfrak{T}^* \times G_{m,r}$ where $\mathfrak{T}^*$ is the binary tetrahedral group of order $24$, and $G_{m,r}$ is either cyclic of order $m$ with $\gcd(m,6)=1$, or of the type (2) with $\gcd(|G_{m,r}|, 6)=1.$ In both cases, for all primes $p~|~m,$ the smallest integer $\gamma_p$ satisfying $2^{\gamma_p} \equiv 1~(\textrm{mod}~p)$ is odd. \\
  (4) $\mathfrak{O}^*$, the binary octahedral group of order $48$. \\
  (5) $\mathfrak{I}^*,$ the binary icosahedral group of order $120.$
\end{theorem}
\begin{proof}
  For a proof, see \cite[Theorem 7]{1}.
\end{proof}

Considering our situation, let $D$ be a division algebra of degree $g$ over its center $K$ (i.e.\ $[D:K]=g^2$) where $K$ is a totally imaginary quadratic number field. If $G$ is a finite subgroup of the multiplicative subgroup of $D,$ then in view of \cite[Corollary 7]{1}, we have that $G$ is either of type (1) or of type (2) in Theorem \ref{thm 12}. Now, if the group $G:=G_{m,r}$ is contained in $D^{\times},$ then $n~|~g$ (see \cite[\S7]{1}), and hence, we have that either $n=1$ or $n=g.$ Furthermore, if $n=g,$ then $g^2$ divides $|G|$ (see \cite[\S7]{1}). This observation leads us to the following

\begin{lemma}\label{cyc lem}
Let $D$ be as in the above and let $G:=G_{m,r}$ for some relatively prime integers $m,r.$ If $n=g$ (so that $G$ is not cyclic) and $|G|$ is even, then $G$ cannot be embedded in $D^{\times}.$
\end{lemma}
\begin{proof}
Note first that $m$ is even. Since $G$ contains an element of order $m,$ we have $\varphi(m)~|~2g,$ and hence, $\varphi(m)\in \{1,2,g,2g\}.$ Then it follows that
\begin{equation*}
m \in
\begin{cases} \{2,4,6,14,18\} & ~~\textrm{if}~g=3; \\
\{2,4,6,4g+2\} & ~~\textrm{if}~g \geq 5~\textrm{and}~2g+1~\textrm{is a prime number;} \\
\{2,4,6\} & ~~\textrm{if}~g \geq 5~\textrm{and}~2g+1~\textrm{is not a prime number.}
\end{cases}
\end{equation*}
We examine each case one by one. \\

(i) If $g=3,$ then we have
\begin{equation*}
|G|=mn \in \{6, 12, 18, 42, 54\}.
\end{equation*}
Since $9$ divides $18$ and $54,$ we have the following two subcases to consider: if $m=6$, then either $s=3, t=2$ or $s=6, t=1$ in view of Theorem \ref{thm 11}. For the first case, since $s=3,$ we have $r \equiv 4~(\textrm{mod}~6)$, which contradicts the assumption that $\textrm{gcd}(m,r)=1.$ For the second case, since $s=6,$ we have $r \equiv 1~(\textrm{mod}~6)$, which contradicts the fact that $n=3.$ Similarly, if $m=18$, then either $s=9, t=2$ or $s=18, t=1.$ For the first case, we have $r \equiv 10~(\textrm{mod}~18)$, which contradicts the assumption that $\textrm{gcd}(m,r)=1.$ For the second case, we have $r \equiv 1~(\textrm{mod}~18)$, which contradicts the fact that $n=3.$ Hence, this case cannot occur. \\

(ii) If $g \geq 5$ and $2g+1$ is a prime number, then we have
\begin{equation*}
|G|=mn \in \{2g, 4g, 6g, (4g+2)g\}.
\end{equation*}
Since $g^2$ does not divide any of these 4 elements, it follows that this case cannot occur. \\

(iii) If $g \geq 5$ and $2g+1$ is not a prime number, then we have
\begin{equation*}
|G|=mn \in \{2g, 4g, 6g\}.
\end{equation*}
Since $g^2$ does not divide any of these 3 elements, it follows that this case cannot occur. \\

This completes the proof.
\end{proof}

In view of \cite[Corollary 7]{1} and Lemma \ref{cyc lem}, we have the following
\begin{theorem}\label{thm 18}
  Let $D$ be as in the above. If a finite group $G$ of even order can be embedded in $D^{\times}$, then $G$ is one of the following groups (up to isomorphism):
  \begin{equation*}
  G \in
\begin{cases} \{\Z/2\Z,\Z/4\Z,\Z/6\Z,\Z/14\Z,\Z/18\Z\} & ~~\textrm{if}~g=3; \\
\{\Z/2\Z,\Z/4\Z,\Z/6\Z,\Z/(4g+2)\Z\} & ~~\textrm{if}~g \geq 5~\textrm{and}~2g+1~\textrm{is a prime number;} \\
\{\Z/2\Z,\Z/4\Z,\Z/6\Z\} & ~~\textrm{if}~g \geq 5~\textrm{and}~2g+1~\textrm{is not a prime number.}
\end{cases}
\end{equation*}
\end{theorem}

Regarding our goal of this paper, the above theorem has a nice consequence:
\begin{corollary}\label{cor 19}
  Let $X$ be a simple abelian variety of dimension $g$ over a finite field $k.$ Let $G$ be a finite subgroup of the multiplicative subgroup of $\textrm{End}_k^0(X).$ Then $G$ is one of the following groups (up to isomorphism): \begin{equation*}
  G \in
\begin{cases} \{\Z/2\Z,\Z/4\Z,\Z/6\Z,\Z/14\Z,\Z/18\Z\} & ~~\textrm{if}~g=3; \\
\{\Z/2\Z,\Z/4\Z,\Z/6\Z,\Z/(4g+2)\Z\} & ~~\textrm{if}~g \geq 5~\textrm{and}~2g+1~\textrm{is a prime number;} \\
\{\Z/2\Z,\Z/4\Z,\Z/6\Z\} & ~~\textrm{if}~g \geq 5~\textrm{and}~2g+1~\textrm{is not a prime number.}
\end{cases}
\end{equation*}
\end{corollary}
\begin{proof}
By Lemma \ref{poss end alg}, $\textrm{End}_k^0(X)$ is either a CM-field of degree $2g$ or a central simple division algebra of degree $g$ over a totally imaginary quadratic field. For the first case, we know that $G$ is cyclic. If $G=\langle f \rangle$ for some element $f$ of order $d,$ then $\varphi(d)~|~2g,$ and hence, $G$ is one of the groups in the above list. Now, if we are in the second case, then $G$ is also one of the groups in the above list by Theorem \ref{thm 18}. \\

This completes the proof.
\end{proof}

\section{Main Result}\label{main}
In this section, we will give a classification of finite groups that can be realized as the automorphism group of a simple polarized abelian variety of dimension $g$ over a finite field. \\

Let $G$ be a finite group. Our main result is the following

\begin{theorem}\label{main thm}
There exists a finite field $k$ and a simple abelian variety $X$ of dimension $g$ over $k$ with a polarization $\mathcal{L}$ such that $G=\textrm{Aut}_k(X,\mathcal{L})$ if and only if $G$ is one of the following groups (up to isomorphism):
  \begin{center}
  \begin{tabular}{|c|c|c|}
\hline
$$ & $G$ & g \\
\hline
$\sharp 1$ & $\Z/2\Z$ & $-$  \\
\hline
$\sharp 2$ & $\Z/4\Z$  & $-$ \\
\hline
$\sharp 3$ & $\Z/6\Z$ & $-$ \\
\hline
$\sharp 4$ & $\Z/14\Z$ &$ g=3 $ \\
\hline
$\sharp 5$ & $\Z/18\Z$ & $g=3$ \\
\hline
$\sharp 6$ & $\Z/(4g+2)\Z$  & $g \geq 5~\textrm{and}~2g+1~\textrm{is a prime number}$ \\
\hline
\end{tabular}
\vskip 4pt
\textnormal{Table 2}
\end{center}
\end{theorem}

\begin{proof}
Suppose first that there exists a finite field $k$ and a simple abelian variety $X$ of dimension $g$ over $k$ with a polarization $\mathcal{L}$ such that $G=\textrm{Aut}_k (X,\mathcal{L}).$ Then by Corollary \ref{cor 19}, $G$ is one of the 6 groups in the above table. Hence, it suffices to show the converse. We prove the converse by considering them one by one. \\

(1) Take $G=\Z/2\Z.$ Let $k=\F_{125}$ and let $\pi$ be a zero of the quadratic polynomial $t^2 +5t+125 \in \Z[t]$ so that $\pi$ is a $125$-Weil number. Then by Theorem \ref{thm HondaTata}, there exists a simple abelian variety $X$ of dimension $r$ over $\F_{125}$ such that $\pi_X$ is conjugate to $\pi,$ and, in fact, we can take $r=3$ by \cite[Proposition 1.2]{3}. Hence, it follows that we have $\Q(\pi_X)=\Q(\sqrt{-19})$ and $D:=\textrm{End}_k^0(X)$ is a central simple division algebra of degree $3$ over $\Q(\sqrt{-19})$ by Theorem \ref{poss end alg}. Now, it is well-known that such $D$ has a maximal $\Z$-order $\mathcal{O}$ containing the ring of integers $\Z \left[\frac{1+\sqrt{-19}}{2} \right]$ of $\Q(\sqrt{-19}).$ Since $\mathcal{O}$ is a maximal $\Z$-order in $D,$ there exists a simple abelian variety $X^{\prime}$ over $k$ such that $X^{\prime}$ is $k$-isogenous to $X$ and $\textrm{End}_k(X^{\prime})=\mathcal{O}$ by Proposition \ref{endposs}. Note also that $\Z/2\Z \cong \langle -1 \rangle \leq \Z \left[\frac{1+\sqrt{-19}}{2} \right]^{\times} \leq \mathcal{O}^{\times}=\textrm{Aut}_k(X^{\prime}).$  \\

  Now, let $\mathcal{L}$ be an ample line bundle on $X^{\prime}$, and put
  \begin{equation*}
    \mathcal{L}^{\prime}:=\bigotimes_{f \in \langle -1 \rangle} f^* \mathcal{L}.
  \end{equation*}
  Then $\mathcal{L}^{\prime}$ is an ample line bundle on $X^{\prime}$ that is preserved under the action of $\langle -1 \rangle$ so that $\langle -1 \rangle \leq \textrm{Aut}_k(X^{\prime},\mathcal{L}^{\prime}).$ Now, we show that $\langle -1 \rangle$ is a maximal finite subgroup of the multiplicative subgroup of $\textrm{End}_k^0(X^{\prime})=D.$ Indeed, suppose that there is a finite subgroup $H$ of $D^{\times}$, which properly contains $\langle -1 \rangle.$ Then $H$ is (isomorphic to) one of the 4 cyclic groups $\Z/4\Z, \Z/6\Z, \Z/14\Z,$ and $\Z/18\Z$ by Corollary \ref{cor 19}. If $H=\Z/4\Z,$ then $\Q(\zeta_4, \sqrt{-19})$ is a subfield of $D,$ which contains $\Q(\sqrt{-19})$, and this is absurd by the double centralizer theorem. Similarly, if $H=\Z/6\Z,$ then $\Q(\zeta_6, \sqrt{-19})$ is a subfield of $D,$ which contains $\Q(\sqrt{-19})$, and this is impossible. Now, if $H=\Z/14\Z,$ then $\Q(\zeta_{14}, \sqrt{-19})$ is a subfield of $D,$ which contains $\Q(\sqrt{-19})$. Then we have $\left[\Q(\zeta_{14}, \sqrt{-19}):\Q(\sqrt{-19})\right]=6 > 3,$ which is absurd. Similarly, if $H=\Z/18\Z,$ then $\Q(\zeta_{18}, \sqrt{-19})$ is a subfield of $D,$ which contains $\Q(\sqrt{-19})$ so that we have $\left[\Q(\zeta_{18}, \sqrt{-19}):\Q(\sqrt{-19})\right]=6 > 3,$ which is impossible. Therefore we can conclude that $\langle -1 \rangle$ is a maximal finite subgroup of $D^{\times}$. Then since $\textrm{Aut}_k(X^{\prime},\mathcal{L}^{\prime})$ is a finite subgroup of $D^{\times},$ it follows that
  \begin{equation*}
    G \cong \langle -1 \rangle =\textrm{Aut}_k(X^{\prime},\mathcal{L}^{\prime}).
  \end{equation*}

(2) Take $G=\Z/4\Z.$ Let $k=\F_{125}$ and let $\pi$ be a zero of the quadratic polynomial $t^2 +10t+125 \in \Z[t]$ so that $\pi$ is a $125$-Weil number. Then by Theorem \ref{thm HondaTata}, there exists a simple abelian variety $X$ of dimension $r$ over $\F_{125}$ such that $\pi_X$ is conjugate to $\pi,$ and, in fact, we can take $r=3$ by \cite[Proposition 1.2]{3}. Hence, it follows that we have $\Q(\pi_X)=\Q(\sqrt{-1})$ and $D:=\textrm{End}_k^0(X)$ is a central simple division algebra of degree $3$ over $\Q(\sqrt{-1})$ by Theorem \ref{poss end alg}. Now, it is well-known that such $D$ has a maximal $\Z$-order $\mathcal{O}$ containing the ring of integers $\Z[\sqrt{-1}]$ of $\Q(\sqrt{-1}).$ Since $\mathcal{O}$ is a maximal $\Z$-order in $D,$ there exists a simple abelian variety $X^{\prime}$ over $k$ such that $X^{\prime}$ is $k$-isogenous to $X$ and $\textrm{End}_k(X^{\prime})=\mathcal{O}$ by Proposition \ref{endposs}. Note also that $\Z/4\Z \cong \langle \sqrt{-1} \rangle \leq \Z[\sqrt{-1}]^{\times} \leq \mathcal{O}^{\times}=\textrm{Aut}_k(X^{\prime}).$  \\

  Now, let $\mathcal{L}$ be an ample line bundle on $X^{\prime}$, and put
  \begin{equation*}
    \mathcal{L}^{\prime}:=\bigotimes_{f \in \langle \sqrt{-1} \rangle} f^* \mathcal{L}.
  \end{equation*}
  Then $\mathcal{L}^{\prime}$ is an ample line bundle on $X^{\prime}$ that is preserved under the action of $\langle \sqrt{-1} \rangle$ so that $\langle \sqrt{-1} \rangle \leq \textrm{Aut}_k(X^{\prime},\mathcal{L}^{\prime}).$ Also, by Corollary \ref{cor 19}, we can see that $\langle \sqrt{-1} \rangle$ is a maximal finite subgroup of the multiplicative subgroup of $\textrm{End}_k^0(X^{\prime})=D$. Then since $\textrm{Aut}_k (X^{\prime},\mathcal{L}^{\prime})$ is a finite subgroup of $D^{\times},$ it follows that
  \begin{equation*}
    G \cong \langle \sqrt{-1} \rangle =\textrm{Aut}_k (X^{\prime}, \mathcal{L}^{\prime}).
  \end{equation*}

(3) Take $G=\Z/6\Z.$ Let $k=\F_{6859}$ and let $\pi$ be a zero of the quadratic polynomial $t^2 +19t+6859 \in \Z[t]$ so that $\pi$ is a $6859$-Weil number. Then by Theorem \ref{thm HondaTata}, there exists a simple abelian variety $X$ of dimension $r$ over $\F_{6859}$ such that $\pi_X$ is conjugate to $\pi,$ and, in fact, we can take $r=3$ by \cite[Proposition 1.2]{3}. Hence, it follows that we have $\Q(\pi_X)=\Q(\sqrt{-3})$ and $D:=\textrm{End}_k^0(X)$ is a central simple division algebra of degree $3$ over $K:=\Q(\sqrt{-3})$ by Theorem \ref{poss end alg}. Now, it is well-known that such $D$ has a maximal $\Z$-order $\mathcal{O}$ containing the ring of integers $\Z \left[\frac{1+\sqrt{-3}}{2} \right]$ of $K.$ Since $\mathcal{O}$ is a maximal $\Z$-order in $D,$ there exists a simple abelian variety $X^{\prime}$ over $k$ such that $X^{\prime}$ is $k$-isogenous to $X$ and $\textrm{End}_k(X^{\prime})=\mathcal{O}$ by Proposition \ref{endposs}. Note also that $\Z/6\Z \cong \left \langle \frac{1+\sqrt{-3}}{2} \right \rangle \leq \Z \left[\frac{1+\sqrt{-3}}{2} \right]^{\times} \leq \mathcal{O}^{\times}=\textrm{Aut}_k(X^{\prime}).$  \\

  Now, let $\mathcal{L}$ be an ample line bundle on $X^{\prime}$, and put
  \begin{equation*}
    \mathcal{L}^{\prime}:=\bigotimes_{f \in \left \langle \frac{1+\sqrt{-3}}{2} \right \rangle} f^* \mathcal{L}.
  \end{equation*}
  Then $\mathcal{L}^{\prime}$ is an ample line bundle on $X^{\prime}$ that is preserved under the action of $\left \langle \frac{1+\sqrt{-3}}{2} \right \rangle$ so that $\left \langle \frac{1+\sqrt{-3}}{2} \right \rangle \leq \textrm{Aut}_k(X^{\prime},\mathcal{L}^{\prime}).$ Now, we show that $\left \langle \frac{1+\sqrt{-3}}{2} \right \rangle$ is a maximal finite subgroup of the multiplicative subgroup of $\textrm{End}_k^0(X^{\prime})=D.$ Indeed, suppose that there is a finite subgroup $H$ of $D^{\times}$, which properly contains $\left \langle \frac{1+\sqrt{-3}}{2} \right \rangle.$ Then $H$ is (isomorphic to) $\Z/18\Z$ by Corollary \ref{cor 19}, and hence, there is an embedding of $L:=\Q(\zeta_{18})$ over $K$ into $D.$ Let $\mathfrak{p}$ be a prime of $K$ lying over $19$ and let $\mathfrak{P}$ be a prime of $L$ lying over $\mathfrak{p}.$ In particular, $\mathfrak{P}$ lies over $19.$ Since $19$ splits completely in both $K$ and $L,$ we have $[L_{\mathfrak{P}}:K_{\mathfrak{p}}]=1.$ On the other hand, it follows from Propositions \ref{local inv} and \ref{index end alg} that $D_{\mathfrak{p}}:= D \otimes_{K} K_{\mathfrak{p}}$ is a central simple division algebra of degree $3$ over $K_{\mathfrak{p}}$. Then since $3$ does not divide $1 = [L_{\mathfrak{P}}:K_{\mathfrak{p}}],$ this contradicts our assumption by the first theorem in \cite[p. 407]{7}. Therefore we can conclude that $\left \langle \frac{1+\sqrt{-3}}{2} \right \rangle$ is a maximal finite subgroup of $D^{\times}$. Then since $\textrm{Aut}_k(X^{\prime},\mathcal{L}^{\prime})$ is a finite subgroup of $D^{\times},$ it follows that
  \begin{equation*}
    G \cong \left \langle \frac{1+\sqrt{-3}}{2} \right \rangle =\textrm{Aut}_k(X^{\prime},\mathcal{L}^{\prime}).
  \end{equation*}

(4) Take $G=\Z/14\Z.$ Let $k=\F_{1331}$ and let $\pi$ be a zero of the quadratic polynomial $t^2 +44t+1331 \in \Z[t]$ so that $\pi$ is a $1331$-Weil number. Then by Theorem \ref{thm HondaTata}, there exists a simple abelian variety $X$ of dimension $r$ over $\F_{1331}$ such that $\pi_X$ is conjugate to $\pi,$ and, in fact, we can take $r=3$ by \cite[Proposition 1.2]{3}. Hence, it follows that we have $\Q(\pi_X)=\Q(\sqrt{-7})$ and $D:=\textrm{End}_k^0(X)$ is a central simple division algebra of degree $3$ over $\Q(\sqrt{-7})$ by Theorem \ref{poss end alg}. Also, by a similar argument as in the proof of (3), we can see that there is an embedding of $\Q(\zeta_{14})$ over $\Q(\sqrt{-7})$ into $D$ in view of the first theorem in \cite[p. 407]{7}. Now, it is well-known that such $D$ has a maximal $\Z$-order $\mathcal{O}$ containing the ring of integers $\Z[\zeta_{14}]$ of $\Q(\zeta_{14}).$ Since $\mathcal{O}$ is a maximal $\Z$-order in $D,$ there exists a simple abelian variety $X^{\prime}$ over $k$ such that $X^{\prime}$ is $k$-isogenous to $X$ and $\textrm{End}_k(X^{\prime})=\mathcal{O}$ by Proposition \ref{endposs}. Note also that $\Z/14\Z \cong \left \langle \zeta_{14} \right \rangle \leq \Z [\zeta_{14}]^{\times} =\Z^2 \times \left \langle \zeta_{14} \right \rangle \leq \mathcal{O}^{\times}=\textrm{Aut}_k(X^{\prime}).$  \\

  Now, let $\mathcal{L}$ be an ample line bundle on $X^{\prime}$, and put
  \begin{equation*}
    \mathcal{L}^{\prime}:=\bigotimes_{f \in \left \langle \zeta_{14} \right \rangle} f^* \mathcal{L}.
  \end{equation*}
  Then $\mathcal{L}^{\prime}$ is an ample line bundle on $X^{\prime}$ that is preserved under the action of $\left \langle \zeta_{14} \right \rangle$ so that $\left \langle \zeta_{14} \right \rangle \leq \textrm{Aut}_k(X^{\prime},\mathcal{L}^{\prime}).$ Also, by Corollary \ref{cor 19}, we can see that $\left \langle \zeta_{14} \right \rangle$ is a maximal finite subgroup of the multiplicative subgroup of $\textrm{End}_k^0(X^{\prime})=D$. Then since $\textrm{Aut}_k (X^{\prime},\mathcal{L}^{\prime})$ is a finite subgroup of $D^{\times},$ it follows that
  \begin{equation*}
    G \cong \left \langle \zeta_{14} \right \rangle =\textrm{Aut}_k (X^{\prime}, \mathcal{L}^{\prime}).
  \end{equation*}

(5) Take $G=\Z/18\Z.$ Let $k=\F_{343}$ and let $\pi$ be a zero of the quadratic polynomial $t^2 +7t+343 \in \Z[t]$ so that $\pi$ is a $343$-Weil number. Then a similar argument as in (4) can be used to show that there is a simple abelian variety $X$ of dimension $3$ over $k$ with a polarization $\mathcal{L}$ such that $G=\textrm{Aut}_k(X,\mathcal{L}).$ \\

(6) Take $G=\Z/(4g+2)\Z$ where $g \geq 5$ and $2g+1$ is a prime number. Let $k=\F_{(2g+1)^2}.$ Then $\pi:=(2g+1) \cdot \zeta_{4g+2}$ is a $(2g+1)^2$-Weil number, and hence, by Theorem \ref{thm HondaTata}, there exists a simple abelian variety $X$ of dimension $r$ over $k$ such that $\pi_X$ is conjugate to $\pi.$ Then we have $[\Q(\pi_X):\Q]=[\Q(\pi):\Q]=2g$ and $\Q(\pi_X)=\Q(\zeta_{4g+2}),$ and hence, $\Q(\pi_X)$ has no real embeddings so that all the local invariants of $\textrm{End}_k^0(X)$ are zero. Hence, it follows from Proposition \ref{index end alg} that $\textrm{End}_k^0(X)=\Q(\pi_X)$ and it is a CM-field of degree $2r$ over $\Q.$ Thus we have $r=g$, and $X$ is a simple abelian variety of dimension $g$ over $k$ with $\textrm{End}_k^0(X)=\Q(\zeta_{4g+2}).$ Now, since $\Z[\zeta_{4g+2}]$ is a maximal $\Z$-order in $\Q(\zeta_{4g+2}),$ there exists a simple abelian variety $X^{\prime}$ over $k$ such that $X^{\prime}$ is $k$-isogenous to $X$ and $\textrm{End}_k(X^{\prime})=\Z[\zeta_{4g+2}]$ by Proposition \ref{endposs}. Note also that $\Z/(4g+2)\Z \cong \langle \zeta_{4g+2} \rangle \leq \Z[\zeta_{4g+2}]^{\times}=\Z^{g-1} \times \langle \zeta_{4g+2} \rangle =\textrm{Aut}_k(X^{\prime}).$  \\

  Now, let $\mathcal{L}$ be an ample line bundle on $X^{\prime}$, and put
  \begin{equation*}
    \mathcal{L}^{\prime}:=\bigotimes_{f \in \langle \zeta_{4g+2} \rangle} f^* \mathcal{L}.
  \end{equation*}
  Then $\mathcal{L}^{\prime}$ is an ample line bundle on $X^{\prime}$ that is preserved under the action of $\langle \zeta_{4g+2} \rangle$ so that $\langle \zeta_{4g+2} \rangle \leq \textrm{Aut}_k(X^{\prime},\mathcal{L}^{\prime}).$ Note also that the maximal finite subgroup of $\Z^{g-1} \times \langle \zeta_{4g+2} \rangle $ is $(\langle -1 \rangle)^{g-1} \times \langle \zeta_{4g+2} \rangle \cong (\Z/2\Z)^{g-1} \times \Z/(4g+2)\Z$ by Goursat's Lemma. Since $\textrm{Aut}_k (X^{\prime},\mathcal{L}^{\prime})$ is a finite subgroup of the multiplicative subgroup of $\Q(\zeta_{4g+2}),$ we know that $\textrm{Aut}_k (X^{\prime}, \mathcal{L}^{\prime})$ is cyclic. Hence, we get
\begin{equation*}
G \cong \langle \zeta_{4g+2} \rangle = \textrm{Aut}_k (X^{\prime}, \mathcal{L}^{\prime}).
\end{equation*}

This completes the proof.
\end{proof}

\begin{remark}\label{alt proof}
We can give alternative proofs of (3), (4), and (5) as follows: \\

(3) Take $G=\Z/6\Z$. Let $k=\F_{7^g}$ where $g \geq 5$ and $2g+1$ is not a prime number. Also, let $\pi$ be a zero of the quadratic polynomial $t^2 +7^{\frac{g-1}{2}} t+7^g \in \Z[t]$ so that $\pi$ is a $7^g$-Weil number. Then by Theorem \ref{thm HondaTata}, there exists a simple abelian variety $X$ of dimension $r$ over $\F_{7^g}$ such that $\pi_X$ is conjugate to $\pi,$ and, in fact, we can take $r=g$ by \cite[11.5]{9}. Hence, it follows that we have $\Q(\pi_X)=\Q(\sqrt{-3})$ and $D:=\textrm{End}_k^0(X)$ is a central simple division algebra of degree $g$ over $\Q(\sqrt{-3})$ by Theorem \ref{poss end alg}. Now, it is well-known that such $D$ has a maximal $\Z$-order $\mathcal{O}$ containing the ring of integers $\Z \left[\frac{1+\sqrt{-3}}{2} \right]$ of $\Q(\sqrt{-3}).$ Since $\mathcal{O}$ is a maximal $\Z$-order in $D,$ there exists a simple abelian variety $X^{\prime}$ over $k$ such that $X^{\prime}$ is $k$-isogenous to $X$ and $\textrm{End}_k(X^{\prime})=\mathcal{O}$ by Proposition \ref{endposs}. Note also that $\Z/6\Z \cong \left \langle \frac{1+\sqrt{-3}}{2} \right \rangle \leq \Z \left[\frac{1+\sqrt{-3}}{2}\right]^{\times} \leq \mathcal{O}^{\times}=\textrm{Aut}_k(X^{\prime}).$  \\

  Now, let $\mathcal{L}$ be an ample line bundle on $X^{\prime}$, and put
  \begin{equation*}
    \mathcal{L}^{\prime}:=\bigotimes_{f \in \left \langle \frac{1+\sqrt{-3}}{2} \right \rangle} f^* \mathcal{L}.
  \end{equation*}
  Then $\mathcal{L}^{\prime}$ is an ample line bundle on $X^{\prime}$ that is preserved under the action of $\left \langle \frac{1+\sqrt{-3}}{2} \right \rangle$ so that $\left \langle \frac{1+\sqrt{-3}}{2} \right \rangle \leq \textrm{Aut}_k(X^{\prime},\mathcal{L}^{\prime}).$ Also, by Corollary \ref{cor 19}, we can see that $\left \langle \frac{1+\sqrt{-3}}{2} \right \rangle$ is a maximal finite subgroup of the multiplicative subgroup of $\textrm{End}_k^0(X^{\prime})=D$. Then since $\textrm{Aut}_k (X^{\prime},\mathcal{L}^{\prime})$ is a finite subgroup of $D^{\times},$ it follows that
  \begin{equation*}
    G \cong \left \langle \frac{1+\sqrt{-3}}{2} \right \rangle =\textrm{Aut}_k (X^{\prime}, \mathcal{L}^{\prime}).
  \end{equation*}

(4) Take $G=\Z/14\Z.$ Let $k=\F_{49}.$ Then $\pi:=7 \cdot \zeta_{14}$ is a $49$-Weil number, and hence, by Theorem \ref{thm HondaTata}, there exists a simple abelian variety $X$ of dimension $r$ over $k$ such that $\pi_X$ is conjugate to $\pi,$ and, in fact, we can take $r=3$ by \cite[\S 5]{3}. Hence, it follows that we have $\Q(\pi_X)=\Q(\zeta_{14})$ and $\textrm{End}_k^0(X)=\Q(\zeta_{14})$ is a CM-field of degree $6$ by Theorem \ref{poss end alg}. Now, since $\Z[\zeta_{14}]$ is a maximal $\Z$-order in $\Q(\zeta_{14}),$ there exists a simple abelian variety $X^{\prime}$ over $k$ such that $X^{\prime}$ is $k$-isogenous to $X$ and $\textrm{End}_k(X^{\prime})=\Z[\zeta_{14}]$ by Proposition \ref{endposs}. Note also that $\Z/14\Z \cong \langle \zeta_{14} \rangle \leq \Z[\zeta_{14}]^{\times}=\Z^2 \times \langle \zeta_{14} \rangle =\textrm{Aut}_k(X^{\prime}).$  \\

  Now, let $\mathcal{L}$ be an ample line bundle on $X^{\prime}$, and put
  \begin{equation*}
    \mathcal{L}^{\prime}:=\bigotimes_{f \in \langle \zeta_{14} \rangle} f^* \mathcal{L}.
  \end{equation*}
  Then $\mathcal{L}^{\prime}$ is an ample line bundle on $X^{\prime}$ that is preserved under the action of $\langle \zeta_{14} \rangle$ so that $\langle \zeta_{14} \rangle \leq \textrm{Aut}_k(X^{\prime},\mathcal{L}^{\prime}).$ Note also that the maximal finite subgroup of $\Z^2 \times \langle \zeta_{14} \rangle $ is $\langle -1 \rangle \times \langle -1 \rangle \times \langle \zeta_{14} \rangle \cong \Z/2\Z \times \Z/2\Z \times \Z/14\Z$ by Goursat's Lemma. Since $\textrm{Aut}_k (X^{\prime},\mathcal{L}^{\prime})$ is a finite subgroup of the multiplicative subgroup of $\Q(\zeta_{14}),$ we know that $\textrm{Aut}_k (X^{\prime}, \mathcal{L}^{\prime})$ is cyclic. Hence, we get
\begin{equation*}
G \cong \langle \zeta_{14} \rangle = \textrm{Aut}_k (X^{\prime}, \mathcal{L}^{\prime}).
\end{equation*}

(5) Take $G=\Z/18\Z.$ Let $k=\F_9$ and let $\pi = 3 \cdot \zeta_{18}$ be a $9$-Weil number. Then a similar argument as in Remark \ref{alt proof}-(4) can be used to show that there is a simple abelian variety $X$ of dimension $3$ over $k$ with a polarization $\mathcal{L}$ such that $G=\textrm{Aut}_k(X,\mathcal{L}).$
\end{remark}

We conclude this section by introducing another interesting result that is related to Theorem \ref{main thm}. Again, let $G$ be a finite group.
\begin{theorem}\label{application}
There exists an algebraically closed field $K$ of characteristic $p>0$ and a simple abelian variety $X$ of dimension $3$ over $K$ with a polarization $\mathcal{L}$ such that $G=\textrm{Aut}_K(X,\mathcal{L})$ if and only if $G$ is one of the following groups (up to isomorphism):
\begin{center}
  \begin{tabular}{|c|c|c|}
\hline
$$ & $G$ \\
\hline
$\sharp 1$ & $\Z/2\Z$   \\
\hline
$\sharp 2$ & $\Z/4\Z$  \\
\hline
$\sharp 3$ & $\Z/6\Z$ \\
\hline
$\sharp 4$ & $\Z/14\Z$  \\
\hline
$\sharp 5$ & $\Z/18\Z$ \\
\hline
\end{tabular}
\vskip 4pt
\textnormal{Table 3}
\end{center}
\end{theorem}
\begin{proof}
Suppose first that there exists an algebraically closed field $K$ of characteristic $p>0$ and a simple abelian variety $X$ of dimension $3$ over $K$ with a polarization $\mathcal{L}$ such that $G=\textrm{Aut}_K (X,\mathcal{L}).$ Let $D=\textrm{End}_K^0(X)$. Then by Remark \ref{oort rmk}, $D$ is of one of the following types: \\
(i) $D=\Q$;\\
(ii) $D$ is a totally real field of degree $3$;\\
(iii) $D$ is a totally imaginary quadratic field;\\
(iv) $D$ is a totally imaginary quadratic extension field of a totally real field of degree $3$; \\
(v) $D$ is a central simple division algebra of degree $3$ over a totally imaginary quadratic field and the $p$-rank of $X$ is $0.$\\

For the cases (i), (ii), (iii), and (iv), $D$ is a field so that $G$ is cyclic. It follows from dimension counting that $G \in \{\Z/2\Z, \Z/4\Z, \Z/6\Z, \Z/14\Z, \Z/18\Z \}.$ For the case (v), $G$ is one of the 5 groups in the above table by Theorem \ref{thm 18}. Hence, it suffices to show the converse. We prove the converse by considering them one by one. \\

  (1) Take $G=\Z/2\Z.$ Let $k=\F_{125}$ and let $X$ be an abelian variety of dimension $3$ over $k$ such that $\pi_X$ is conjugate (as $125$-Weil numbers) to a zero of the quadratic polynomial $t^2 +5t+125 \in \Z[t]$, and that $\textrm{End}_k(X)$ is a maximal $\Z$-order, containing $\Z \left[ \frac{1+\sqrt{-19}}{2}\right]$, in a central simple division algebra $\textrm{End}_k^0 (X)$ of degree $3$ over $\Q(\sqrt{-19})$ (as in the proof of Theorem \ref{main thm}-(1)). By \cite[Propostion 3]{5}, $X_{\overline{k}}:=X \times_k \overline{k}$ is simple over $\overline{k}.$ Now, it is clear that $\textrm{End}_k(X) \subseteq \textrm{End}_{\overline{k}}(X_{\overline{k}}).$ Then since $\textrm{End}_{\overline{k}}(X_{\overline{k}})$ is a $\Z$-order in $\textrm{End}_{\overline{k}}^0(X_{\overline{k}})= \textrm{End}_k^0(X)$ and $\textrm{End}_k(X)$ is a maximal $\Z$-order in $\textrm{End}_k^0(X),$ it follows that $\textrm{End}_k(X)=\textrm{End}_{\overline{k}}(X_{\overline{k}}).$ Note also that $\Z/2\Z \cong \langle -1 \rangle \leq \Z \left[\frac{1+\sqrt{-19}}{2}\right]^{\times} \leq \textrm{End}_{\overline{k}}(X_{\overline{k}})^{\times} = \textrm{Aut}_{\overline{k}}(X_{\overline{k}}).$ \\

  Now, let $\mathcal{L}$ be an ample line bundle on $X_{\overline{k}}$, and put
  \begin{equation*}
    \mathcal{L}^{\prime}:=\bigotimes_{f \in \langle -1 \rangle} f^* \mathcal{L}.
  \end{equation*}
  Then $\mathcal{L}^{\prime}$ is an ample line bundle on $X_{\overline{k}}$ that is preserved under the action of $\langle -1 \rangle$ so that $\langle -1 \rangle \leq \textrm{Aut}_{\overline{k}}(X_{\overline{k}},\mathcal{L}^{\prime}).$ Also, since we have seen that $\langle -1 \rangle$ is a maximal finite subgroup of the multiplicative subgroup of $\textrm{End}_k^0(X)=\textrm{End}_{\overline{k}}^0(X_{\overline{k}})$ in the proof of Theorem \ref{main thm}-(1), it follows that
  \begin{equation*}
    G \cong \langle -1 \rangle =\textrm{Aut}_{\overline{k}}(X_{\overline{k}}, \mathcal{L}^{\prime}).
  \end{equation*}

  (2) Take $G=\Z/4\Z.$ Let $k=\F_{125}$ and let $X$ be an abelian variety of dimension $3$ over $k$ such that $\pi_X$ is conjugate (as $125$-Weil numbers) to a zero of the quadratic polynomial $t^2 +10t+125 \in \Z[t]$, and that $\textrm{End}_k(X)$ is a maximal $\Z$-order, containing $\Z \left[ \sqrt{-1} \right]$, in a central simple division algebra $\textrm{End}_k^0 (X)$ of degree $3$ over $\Q(\sqrt{-1})$ (as in the proof of Theorem \ref{main thm}-(2)). Then a similar argument as in (1) can be used to show that there is a polarization $\mathcal{L}$ on $X_{\overline{k}}$ such that $G=\textrm{Aut}_{\overline{k}}(X_{\overline{k}},\mathcal{L}).$ \\

  (3) Take $G=\Z/6\Z.$ Let $k=\F_{6859}$ and let $X$ be an abelian variety of dimension $3$ over $k$ such that $\pi_X$ is conjugate (as $6859$-Weil numbers) to a zero of the quadratic polynomial $t^2 +19t+6859 \in \Z[t]$, and that $\textrm{End}_k(X)$ is a maximal $\Z$-order, containing $\Z \left[ \frac{1+\sqrt{-3}}{2} \right]$, in a central simple division algebra $\textrm{End}_k^0 (X)$ of degree $3$ over $\Q(\sqrt{-3})$ (as in the proof of Theorem \ref{main thm}-(3)). Then a similar argument as in (1) can be used to show that there is a polarization $\mathcal{L}$ on $X_{\overline{k}}$ such that $G=\textrm{Aut}_{\overline{k}}(X_{\overline{k}},\mathcal{L}).$ \\

  (4) Take $G=\Z/14\Z.$ Let $k=\F_{1331}$ and let $X$ be an abelian variety of dimension $3$ over $k$ such that $\pi_X$ is conjugate (as $1331$-Weil numbers) to a zero of the quadratic polynomial $t^2 +44t+1331 \in \Z[t]$, and that $\textrm{End}_k(X)$ is a maximal $\Z$-order, containing $\Z \left[ \zeta_{14} \right]$, in a central simple division algebra $\textrm{End}_k^0 (X)$ of degree $3$ over $\Q(\sqrt{-7})$ (as in the proof of Theorem \ref{main thm}-(4)). Then a similar argument as in (1) can be used to show that there is a polarization $\mathcal{L}$ on $X_{\overline{k}}$ such that $G=\textrm{Aut}_{\overline{k}}(X_{\overline{k}},\mathcal{L}).$ \\

  (5) Take $G=\Z/18\Z.$ Let $k=\F_{343}$ and let $X$ be an abelian variety of dimension $3$ over $k$ such that $\pi_X$ is conjugate (as $343$-Weil numbers) to a zero of the quadratic polynomial $t^2 +7t+343 \in \Z[t]$, and that $\textrm{End}_k(X)$ is a maximal $\Z$-order, containing $\Z \left[ \zeta_{18} \right]$, in a central simple division algebra $\textrm{End}_k^0 (X)$ of degree $3$ over $\Q(\sqrt{-3})$ (as in the proof of Theorem \ref{main thm}-(5)). Then a similar argument as in (1) can be used to show that there is a polarization $\mathcal{L}$ on $X_{\overline{k}}$ such that $G=\textrm{Aut}_{\overline{k}}(X_{\overline{k}},\mathcal{L}).$ \\

  This completes the proof.
\end{proof}


\begin{thebibliography}{00}
\bibitem{1}
{Shimshon A. Amitsur,} Finite Subgroups of Division Rings, Trans. Amer. Math. Soc. \textrm{80} (1955), 361-386.
\bibitem{2}
{Bas Edixhoven, Gerard van der Geer, and Ben Moonen,} Abelian Varieties, gerard.vdgeer.net/AV.pdf
\bibitem{3}
{Safia Haloui,} The characteristic polynomials of abelian varieties of dimensions 3 over finite fields, J. Number Theory 130, (2010), 2745-2752.
\bibitem{4}
{Taira Honda,} Isogeny classes of abelian varieties over finite fields, J. Math. Soc. Japan Vol 20, (1968), 83-95.
\bibitem{5}
{Everett W. Howe and Hui June Zhu,} On the Existence of Absolutely Simple Abelian Varieties of a Given Dimension over an Arbitrary Field, J. Number Theory 92, (2002), 139-163
\bibitem{6}
{WonTae Hwang,} On a classification of the automorphism groups of polarized abelian surfaces over finite fields, arXiv:1809.06251
\bibitem{7}
{Benjamin Linowitz and Thomas Shemanske,} Embedding orders into central simple algebras, J. Théor. Nr. Bordx., \textrm{24} (2012), 405-424.
\bibitem{8}
{David Mumford,} Abelian Varieties, Hindustan Book Agency (2008), Corrected reprint of the 2nd edition.
\bibitem{9}
{Frans Oort,} Abelian varieties over finite fields, Higher-dimensional varieties over finite fields, Summer school in G\"ottingen, (2007), 1-67.
\bibitem{10}
{Frans Oort,} Endomorphism Algebras of Abelian Varieties, Algebraic Geometry and Commutative Algebra in Honor of Masayoshi NAGATA, (1987), 469-502.
\bibitem{11}
{John Tate,} Endomorphisms of Abelian Varieties over Finite Fields, Invent. Math., \textrm{2} (1966), 131-144.
\bibitem{12}
{William C. Waterhouse,} Abelian varieties over finite fields, Annales scientifiques de l'E.N.S. $4^e$ serie, tome 2, $n^{\circ}$ (1969), 521-560.
\end{thebibliography}
\end{document}